\documentclass{jmsj}
\usepackage{amsthm}
\usepackage{amssymb}
\usepackage{lmodern}
\usepackage{mathtools}

\newcommand{\N}{\mathbb{N}}

\newcommand{\F}{\mathbb{F}}

\newcommand{\Sym}{\mathbb{S}}

\newcommand{\Inn}{\mathrm{Inn}}
\newcommand{\Aut}{\mathrm{Aut}}

\makeatletter
\numberwithin{equation}{section}
\numberwithin{figure}{section}
\theoremstyle{plain}
\newtheorem{thm}{Theorem}
\theoremstyle{plain}
\newtheorem{lem}[thm]{Lemma}
\newtheorem{cor}[thm]{Corollary}
\theoremstyle{plain}
\newtheorem{pro}[thm]{Proposition}
\newtheorem*{conjecture*}{Conjecture}
\newtheorem{problem}[thm]{Problem}
\theoremstyle{remark}

\theoremstyle{plain}
\newtheorem{exa}[thm]{Example}

\makeatother

\begin{document}

\title{Doubly transitive groups and cyclic quandles}

\begin{abstract}
    We prove that for $n>2$ there exists a quandle of cyclic type of size $n$
    if and only if $n$ is a power of a prime number.  This establishes a
    conjecture of S. Kamada, H. Tamaru and K. Wada.  As a corollary, every
    finite quandle of cyclic type is an Alexander quandle.  We also prove that
    finite doubly transitive quandles are of cyclic type. This establishes a
    conjecture of H. Tamaru.  
\end{abstract}

\author{Leandro Vendramin}
\address{
Departamento de Matem\'atica -- FCEN,
Universidad de Buenos Aires, Pab. I -- Ciudad Universitaria (1428)
Buenos Aires -- Argentina}
\email{lvendramin@dm.uba.ar}

\maketitle

\section*{Introduction}

Quandles are algebraic structures deeply related to the Reidemeister moves of
classical knots.  These structures play an important role in knot theory
because they produce strong knot invariants, see for example \cite{MR1990571},
\cite{MR3079760} and \cite{MR3253967}. The applications of quandles in knot theory
force us to study certain particular classes of quandles. One of these classes
is the class of finite quandles of cyclic type. The idea of studying such
quandles goes as far as \cite{MR2194774}. Quandles of cyclic type were
independently considered in \cite{MR3169460} and \cite{MR3127819}.

In this note we present the proofs of two conjectures related to quadles of
cyclic type.  First we prove the following theorem, conjectured by S.  Kamada,
H.  Tamaru and K. Wada, see  \cite[Conjecture~4.7]{KTW}.

\begin{thm}
    \label{thm:KTW}
    Let $n\geq3$. Then there exists a quandle of size $n$ of cyclic type if and
    only if $n$ is a power of a prime number. 
\end{thm}

K. Wada independently proved that cyclic quandles with a prime power size are
Alexander quandles \cite{MR3379001}.  Theorem~\ref{thm:KTW} yields the
following stronger result.

\begin{cor}
    \label{cor:cyclic_is_affine}
		Let $X$ be a finite quandle of cyclic type. Then $|X|$ is a power of a
		prime number and $X$ is an Alexander simple quandle over the field with
		$|X|$ elements. 
\end{cor}

Finally, using the classification of simple groups 
we prove the following theorem. 

\begin{thm}
    \label{thm:2transitive_is_alexander}
		Every finite doubly transitive quandle is an Alexander simple quandle.
\end{thm}

The theorem gains in interest if we know that doubly transitive Alexander
quandles are of cyclic type. This was proved by K. Wada \cite{MR3379001}.  Then
one immediately obtains the following corollary, which proves a conjecture of
H.  Tamaru, see \cite[Conjecture~5.1]{MR3127819}.

\begin{cor}
    \label{cor:2transitive_is_cyclic}
		Every finite doubly transitive quandle is of cyclic type.
\end{cor}

The principal significance of the corollary is that it advances the
classification of $k$-transitive quandles for $k\geq2$. On the other hand, the
classification of finite indecomposable quandles is somewhat out of reach. Thus
the following seems to be an interesting problem. 

\begin{problem}
    Classify finite primitive quandles. 
\end{problem}

The paper is organized as follows. In Section~\ref{section:preliminaries} we
set up notations and terminology, and we review some basic facts about quandles
and permutation groups. Section~\ref{section:KTW} is devoted to prove
Theorem~\ref{thm:KTW} and Corollary~\ref{cor:cyclic_is_affine}. The proof of
the theorem is based on the following observation: the inner group of a finite
quandle of cyclic type is a Frobenius group. The proof of the corollary uses
Theorem~\ref{thm:KTW} and the classification of simple quandles of
Andruskiewitsch and Gra\~na \cite[\S3]{MR1994219}.  In
Section~\ref{section:2transitive_is_alexander} we prove
Theorem~\ref{thm:2transitive_is_alexander}. 
The proof depends on the
classification of simple groups. 

\section{Preliminaries}
\label{section:preliminaries}

Recall that a \emph{quandle} is a set $X$ with a binary operation
$\triangleright\colon X\times X\to X$ such that $x\triangleright x=x$ for all
$x\in X$, the map $\varphi_x\colon X\to X$, $y\mapsto x\triangleright y$, is
bijective for all $x\in X$, and $x\triangleright (y\triangleright
z)=(x\triangleright y)\triangleright (x\triangleright z)$ for all $x,y,z\in X$.
The \emph{inner group} of $X$ is the group $\Inn(X)=\langle \varphi_x\mid x\in
X\rangle$.  The quandle $X$ is \emph{indecomposable} (or \emph{connected}) if
$\Inn(X)$ acts transitively on $X$. From the definition of quandle one
immediately obtains the following lemma.

\begin{lem}
    \label{lem:nontrivial_center}
    Let $X$ be a quandle and $x\in X$. Then $\varphi_x$ is a
    central element of the stabilizer of $x$ in $\Inn(X)$.
\end{lem}

A quandle $X$ is \emph{primitive} if $\Inn(X)$ acts primitively on $X$.  For
$k\geq1$ we say that $X$ is \emph{$k$-transitive} if $\Inn(X)$ acts
$k$-transitively on $X$. It is worth pointing out that $1$-transitive means
indecomposable, and that  $2$-transitive (or \emph{doubly transitive}) quandles
are called two-point homogeneous in \cite{MR3127819}.  A similar argument to
that of \cite[Thm.~9.6]{MR0183775} shows that  doubly transitive quandles are
primitive. Similarly, $(k+1)$-transitive quandles are $k$-transitive for all
$k\geq1$. The following result of McCarron \cite[Prop.~5]{smallhomogeneous}
shows that higher transitivity is a rare phenomenon: the dihedral quandle with
three elements is the unique $3$-transitive quandle.

\begin{lem}[McCarron]
    \label{lem:mccarron}
    Let $k\in\N$ with $k\geq2$ and $X$ be a finite $k$-transitive quandle with
    at least four elements. Then $k\leq2$.
\end{lem}

\begin{proof}
		Suppose that $k\geq3$.  Since $X$ is $k$-transitive, it is indecomposable
		and nontrivial.  Thus let $x,y\in X$ such that $|\{x,y,x\triangleright
		y\}|=3$. By assumption, there exists $z\in X\setminus\{x,y,x\triangleright
		y\}$. Since $\Inn(X)$ acts $k$-transitively on $X$ and $k\geq3$, there
		exists $f\in\Inn(X)$ such that $f(x)=x$, $f(y)=y$ and $f(x\triangleright
		y)=z$. Then $x\triangleright y=f(x)\triangleright f(y)=f(x\triangleright
		y)=z$, a contradiction.
\end{proof}

We shall also need the following lemma of \cite{McCarron}. Recall that a
quandle is \emph{simple} if it has no quotients except itself and the
trivial quandle of one element \cite{MR682881}.

\begin{lem}[McCarron]
    \label{lem:primitive_is_simple}
    Let $X$ be a finite quandle and suppose that $\Inn(X)$ acts primitively on
    $X$. Then $X$ is simple. 
\end{lem}

\begin{proof}
    Suppose that $X$ is not simple. Then there exist a nontrivial quandle
    $Q\ne X$ and a surjective homomorphism of quandles $p\colon X\to Q$.
    Consider the equivalence relation over $X$ given by $x\equiv y$ if and only
    if $p(x)=p(y)$.  We claim that the orbits of this action form a system of
    blocks for $\Inn(X)$. To prove our claim let $x\in X$ and 
		\[
		\Delta_{x}=\{y\in
    X\mid p(x)=p(y)\}
		\]
		be an equivalence class. Then
    $\varphi_y\cdot\Delta_x=\Delta_{\varphi_y(x)}$ for all $y\in X$ and hence
    $f\cdot\Delta_x=\Delta_{f(x)}$ for all $f\in\Inn(X)$.  Thus
    $f\cdot\Delta_x$ is also an equivalence class and therefore
    $f\cdot\Delta_x\cap \Delta_x=\emptyset$ or $f\cdot \Delta_x=\Delta_x$. This
    implies that $\Inn(X)$ is not primitive.
\end{proof}

Following \cite[Definition~3.5]{MR3127819}, we say that a quandle $X$ is of
\emph{cyclic type} (or \emph{cyclic}) if for each $x\in X$ the permutation
$\varphi_x$ acts on $X\setminus\{x\}$ as a cycle of length $|X|-1$, where $|X|$
denotes the cardinality of $X$. Tamaru proved that quandles of cyclic type are
doubly transitive \cite[Prop.~3.6]{MR3127819}. In particular, quandles of
cyclic type are indecomposable.

\begin{exa}[Alexander quandles]
    \label{exa:alexander}
		Alexander quandles form an important family of examples.  Let $A$ be an
		abelian group and $g\in\Aut(A)$. Then $A$ is a quandle with
		$x\triangleright y=(1-g)(x)+g(y)$ for all $x,y\in A$. This is 
		the \emph{Alexander quandle} of type $(A,g)$. 
\end{exa}

\begin{exa}
    Let us mention a particular case of Example~\ref{exa:alexander}.  Let $p$
    be a prime number, $m\in\N$, $q=p^m$, and $\F_q$ be the field of $q$
    elements. For each $\alpha\in\F_q$ the \emph{Alexander quandle} of type
    $(q,\alpha)$ is the quandle structure over $\F_q$ given by $x\triangleright
    y=(1-\alpha)x+\alpha y$ for all $x,y\in\F_q$.  
\end{exa}

\section{Proofs of Theorem~\ref{thm:KTW} and Corollary~\ref{cor:cyclic_is_affine}}
\label{section:KTW}

Using Alexander quandles, H. Tamaru proved the existence of quandles of cyclic
type with a prime number of elements, see \cite[\S4]{MR3127819}.  We use
Tamaru's method  to prove a similar result.  

Recall that for any power $q$ of a prime number, the multiplicative subgroup of
$\F_q$ is cyclic of order $q-1$.  

\begin{pro}
    \label{pro:affine_cyclic}
    Let $p$ be a prime number, $m\in\N$ and $q=p^m$. Let $\alpha\in\F_q$ and
    $X$ be an Alexander quandle of type $(q,\alpha)$. Then $X$ is of cyclic
    type if and only if $\alpha$ has order $q-1$. 
\end{pro}

\begin{proof}
    Suppose first that $X$ is of cyclic type. Then $\varphi_0$ acts on
    $X\setminus\{0\}$ as a cycle of length $q-1$. Thus 
    \[
    \varphi_0=\left(1\;\varphi_0(1)\;\varphi_0^2(1)\cdots\varphi_0^{q-2}(1)\right)
    \] 
    and $\varphi_0^i(1)\ne\varphi_0^j(1)$ for $i,j\in\{0,\dots,q-2\}$ with
    $i\ne j$.  Since $\varphi_0^k(1)=\alpha^k$ for all $k\in\{0,\dots,q-2\}$,
    the claim follows.

    Conversely, suppose that $\alpha$ has order $q-1$.  Since $X$ has no
    nontrivial subquandles by \cite[Prop.~4.1]{MR2799090}, it follows that $X$
    is indecomposable.  The permutation $\varphi_0$ acts on $X$ as the cycle
    $(1\,\alpha\,\alpha^2\cdots\alpha^{q-2})$ of length $q-1$.  Since $X$ is
    indecomposable, this implies that $X$ is of cyclic type by
    \cite[Prop.~3.9]{MR3127819}.
\end{proof}

Now we prove that the cardinality of a finite quandle of cyclic type is some
power of a prime number.  For that purpose, we need some basic properties of
Frobenius groups.  A finite group $G$ acting on a finite set $X$ is a
\emph{Frobenius group} if $G_x\cap G_y=1$ for all $x,y\in X$ with $x\ne y$,
where $G_x$ and $G_y$ denote the stabilizer (or isotropy) subgroups of $x$ and
$y$ respectively. The \emph{degree} of $G$ is the cardinality of $X$.  

It follows from the definition that the center of a Frobenius group is trivial.
The following result is a consequence of \cite[Thm.~5.1]{MR0183775} and
\cite[Thm.~11.3(a)]{MR0183775}.

\begin{thm}
    \label{thm:n=p^m}
    Let $G$ be a doubly transitive Frobenius group of degree $n$. Then $n=p^m$ for
    some prime number $p$ and $m\in\N$.    
\end{thm}

We shall also need the following two lemmas.

\begin{lem}
    \label{lem:cyclic_stabilizer}
    Let $X$ be a finite quandle of cyclic type, $x\in X$, and $G=\Inn(X)$. Then
    $G_x$ is cyclic and generated by $\varphi_x$. 
\end{lem}

\begin{proof}
    Assume that $X$ has $n$ elements. Then $G$ is a subgroup of $\Sym_n$. Since 
    \[
        f\varphi_xf^{-1}=\varphi_{f(x)}=\varphi_x
    \]
    for all $f\in G_x$, we conclude that $G_x\subseteq
    C_G(\varphi_x)$, where $C_G(\varphi_x)$ denotes the centralizer of
    $\varphi_x$ in $G$. The permutation $\varphi_x$ is a cycle of length $n-1$.
    Hence \[
		C_G(\varphi_x)=C_{\Sym_n}(\varphi_x)\cap G=\langle\varphi_x\rangle
		\]
    and therefore $G_x=\langle\varphi_x\rangle$.
\end{proof}

\begin{lem}
    \label{lem:frobenius}
    Let $n\geq3$ and $X$ be a quandle of cyclic type of size $n$. Then
    $\Inn(X)$ is a Frobenius group of degree $n$.
\end{lem}

\begin{proof}
    Let $G=\Inn(X)$ and $x\in X$. By Lemma~\ref{lem:cyclic_stabilizer},
    $G_x=\langle\varphi_x\rangle$.  We claim that for each $g\in G\setminus
    G_x$ the subgroups $G_x$ and $gG_xg^{-1}=G_{g(x)}$ have trivial
    intersection. Let $h\in G_x\cap gG_xg^{-1}$ and assume that
    $h=g\varphi_x^kg^{-1}=\varphi_x^l$ for some $k,l\in\{0,\dots,n-2\}$ and
    $g\in G\setminus G_x$. Then
     \[
        \varphi_x^l=g\varphi_x^kg^{-1}=(g\varphi_xg^{-1})^k=\varphi_{g(x)}^k.
    \]
    Let $y\in X\setminus\{x\}$ such that $g(x)=y$. Then
    $\varphi_x^l=\varphi_y^k$. Since $\varphi_y$ is a $(n-1)$-cycle that
    fixes $y$ and $\varphi_y^k(x)=\varphi_x^l(x)=x$, we conclude that $k=0$.
    From this the claim follows.
\end{proof}

Now we prove that for $n\geq3$ there exists a quandle of cyclic type of size $n$
if and only if $n$ is a power of a prime number.  This establishes
\cite[Conjecture~4.7]{KTW}.

\begin{proof}[Proof of Theorem~\ref{thm:KTW}] 
    Assume that $n=p^m$, where $p$ is a prime number and $m\in\N$.  By
    Proposition~\ref{pro:affine_cyclic}, there exists a quandle of cyclic type
    of size $n$.  Conversely, if $X$ is a quandle of cyclic type and size $n$,
    then $\Inn(X)$ is a Frobenius group by Lemma~\ref{lem:frobenius}. Since
    $\Inn(X)$ acts doubly transitively on $X$ by \cite[Prop.~3.6]{MR3127819},
    Theorem~\ref{thm:n=p^m} implies that $n$ is a power of a prime number.
\end{proof}

Theorem~\ref{thm:KTW}, Lemma~\ref{lem:primitive_is_simple} and the
classification of simple quandles of Andruskiewitsch and Gra\~na
\cite[\S3]{MR1994219} yield Corollary~\ref{cor:cyclic_is_affine}.

\begin{proof}[Proof of Corollary~\ref{cor:cyclic_is_affine}]
    Let us assume that $X$ is a cyclic quandle. By Theorem~\ref{thm:KTW}, the
    cardinality of $X$ is some power of a prime number. Since $X$ is
    doubly transitive by \cite[Prop.~3.6]{MR3127819}, it follows that $\Inn(X)$
    acts primitively on $X$. By Lemma~\ref{lem:primitive_is_simple}, $X$
    is simple.  Now \cite[Thm.~3.9]{MR1994219} yields the claim.
\end{proof}
 
\section{Proof of Theorem~\ref{thm:2transitive_is_alexander}}
\label{section:2transitive_is_alexander}

Recall that a \emph{minimal normal} subgroup of $G$ is a normal subgroup $N$ of
$G$ such that $N\ne1$ and $N$ contains no normal subgroup of $G$ except $1$ and
$N$.  The \emph{socle} of $G$ is the product of the minimal normal subgroups of
$G$.  The following theorem goes back to Burnside, see for example
\cite[Thm.~4.3]{MR599634}.

\begin{thm}[Burnside]
    \label{thm:burnside}
    Let $G$ be a doubly transitive group and $N$ be a minimal normal subgroup
    of $G$. Then $N$ is either a regular elementary abelian group, or a
    nonregular nonabelian simple group. 
\end{thm}

As a consequence of the odd analogue of Glauberman Z*-theorem, we prove
that finite doubly transitive quandles are Alexander simple.  The key step is a
group-theoretical result kindly comunicated to us by G. Robinson, see
\verb+http://mathoverflow.net/questions/184682+.  We refer to~\cite[Chapter
6]{MR0231903} for more details.

\begin{proof}[Proof of Theorem~\ref{thm:2transitive_is_alexander}]
	Let $G=\Inn(X)$.  The quandle $X$ is doubly transitive and hence $G$ acts
	primitively on $X$. Then $X$ is simple by
	Lemma~\ref{lem:primitive_is_simple} and therefore $X$ is a conjugacy class
	of $G$ and $G$ has a trivial center by \cite[Lemma~1]{MR682881}. 

	Suppose that $G$ is nonsolvable.  Let $N$ be the commutator subgroup of
	$G$.  Since $N$ is the unique minimal normal subgroup of $G$ by
	\cite[Lemma~2]{MR682881} and $G$ is nonsolvable, it follows from
	Theorem~\ref{thm:burnside} that $N$ is a nonregular nonabelian simple
	group. Hence $G$ is equivalent to a doubly transitive group with simple
	socle. Such groups are classified, see~\cite[Table 7.4]{MR599634}. With
	Lemma~\ref{lem:mccarron} one excludes from~\cite[Table 7.4]{MR599634} the
	groups with transitivity $\geq3$.  Thus we may assume that $G$ is a doubly
	transitive group with $F(G)=1$, where $F(G)$ denotes the Fitting subgroup
	of $G$.  We claim that $Z(G_x)=1$. Suppose that $Z(G_x)\ne1$. Let $p$ be a
	prime number dividing the order of $Z(G_x)$ and let $g\in Z(G_x)$ be an
	element of order $p$.  The case where $p=2$ follows from Glauberman
	Z*-theorem~\cite{MR0202822}, so we may assume that $p$ is odd. The
	permutation action of $G$ on $X$ is
	equivalent to that of $G$ on the conjugacy class of $g$ by conjugation.
	Since the action is doubly transitive and $F(G)=1$, no conjugate of $g$
	other that itself commutes with $g$.  By~\cite[Cor. 2]{MR0197571}, there
	must be a $p'$-subgroup $T$ of $G$ which is normalized, but not centralized
	by $g$.  Hence for some $t \in T$, $g^{-1}t^{-1}gt$ is a nontrivial
	$p'$-element (i.e. an element of order not a multiple of $p$).  Since
	$C_G(g)$ is transitive on the remaining conjugates of $g$, one obtains that
	for all $h\in G$ the order of $g^{-1}h^{-1}gh$ is not a multiple of $p$.
	By~\cite[Thm. D]{MR1239051}, $O_{p'}(G)\ne1$, where $O_{p'}(G)$ denotes the
	largest normal subgroup of $G$ which is a $p'$-group.  Since $G$ is doubly
	transitive on $X$, it follows that $G=O_{p'}(G)C_G(g)$.  From $g\not\in
	Z(G)$ one obtains that there exists a prime number $q\ne p$ and a $q$-Sylow
	subgroup $Q$ of $O_{p'}(G)$ normalized but not centralized by $g$. A
	similar argument and \cite[Thm. D]{MR1239051} prove that
	$[G,g]=[O_{p'}(G),g]$ is a nontrivial $q$-group of $G$ and thus
	$O_q(G)\ne1$, which is a contradiction.

	Since $G$ is the inner group of a quandle, we conclude from
	Lemma~\ref{lem:nontrivial_center} and the previous argument that $G$ is
	solvable.  Since $X$ is a simple quandle and its inner group $G$ is solvable,
	there exists a prime number $p$ and $m\in\N$ such that $X$ is an Alexander
	quandle of size $p^m$ by \cite[Thm.~3.9]{MR1994219}.
\end{proof} 

\section*{Acknowledgement}

The author thanks E. Clark, S. Kamada, A. Lochmann, J. McCarron, H.  Tamaru and
K. Wada for several helpful comments. Special thanks go to G. Bianco for
interesting conversations and to G. Robinson for the group-theoretic argument
used in the proof of Theorem~\ref{thm:2transitive_is_alexander}.  This work is
supported by Conicet, UBACyT 20020110300037, ICTP, and the Alexander von
Humboldt Foundation. 

\def\cprime{$'$}

\end{document}